\documentclass{amsart}
\usepackage[utf8]{inputenc}
\usepackage[T1]{fontenc}
\usepackage[english]{babel}
\usepackage{amsmath, stackrel}
\usepackage{amssymb}
\usepackage{amsthm}
\usepackage{mathrsfs}
\usepackage{mathtools} 
\usepackage{enumerate}
\usepackage{graphicx}
\usepackage{geometry}
\usepackage{color}
\usepackage{ bbold } 
\usepackage[makeroom]{cancel}
\usepackage{todonotes}
\newcommand{\note}[1]{\textcolor{red}{#1}}

\theoremstyle{plain}
\newtheorem{thm}{Theorem}[section]
\newtheorem{lemma}[thm]{Lemma}
\newtheorem{corollary}[thm]{Corollary}
\newtheorem{prop}[thm]{Proposition}
\theoremstyle{definition}
\newtheorem{defn}[thm]{Definition}

\newtheorem{rmk}[thm]{Remark}

\newcommand{\R}{\mathbb{R}}
\newcommand{\Rn}{\mathbb{R}^n}

\newcommand{\N}{\mathbb{N}}

\newcommand{\Sp}{\mathbb{S}}
\newcommand{\sus}{\subseteq}
\newcommand{\inv}[1]{{#1}^{-1}}
\newcommand{\D}{\mathrm{D}}
\newcommand{\der}{\mathrm{d}}

\newcommand{\g}{\mathfrak{g}}
\newcommand{\G}{\mathbb{G}}
\newcommand{\ye}{\mathbb{1}_E}
\newcommand{\Int}{\operatorname{int}}

\newcommand{\at}[1]{\raise-.5ex\hbox{\ensuremath|}_{#1}}

\begin{document}
\title{A rectifiability result for finite-perimeter sets in Carnot groups}
\author{Sebastiano Don}
\address{\textsc{Sebastiano Don}:  Department of Mathematics and Statistics, P.O.\ Box 35 (MaD), FI-40014, University of Jyv\"askyl\"a, Finland.}
\email{sedon@jyu.fi}
\author{Enrico Le Donne}
\address{\textsc{Enrico Le Donne}: 
Dipartimento di Matematica, Universit\`a di Pisa, Largo B. Pontecorvo 5, 56127 Pisa, Italy \\
\& \\
Department of Mathematics and Statistics, P.O.\ Box 35 (MaD), FI-40014, University of Jyv\"askyl\"a, Finland.} \email{enrico.ledonne@unipi.it}
\author{Terhi Moisala}
\address{\textsc{Terhi Moisala}:  Department of Mathematics and Statistics, P.O.\ Box 35 (MaD), FI-40014, University of Jyv\"askyl\"a, Finland.}
\email{tekamois@jyu.fi}
\author{Davide Vittone}
\address{\textsc{Davide Vittone}:  Dipartimento di Matematica Tullio Levi-Civita, Universit\`a di Padova, via Trieste 63, Italy.}
\email{vittone@math.unipd.it}



\thanks{S.D., E.L.D. and T.M.\ have been partially supported by the Academy of Finland
(grant 288501 ``\emph{Geometry of subRiemannian groups}'')
and by the European Research Council
(ERC Starting Grant 713998 GeoMeG ``\emph{Geometry of Metric Groups}''). T.M. was also partially supported by the Finnish Cultural Foundation (grant 00170709). D.V.\ is supported by the University of Padova STARS Project ``Sub-Riemannian Geometry and Geometric Measure Theory Issues: Old and New'' (SUGGESTION), and by GNAMPA of INdAM (Italy) project “Applicazioni della Teoria delle Correnti all’Analisi Reale e al Trasporto Ottimo”.}

\begin{abstract}
	In the setting of Carnot groups, we are concerned with the
	rectifiability problem for subsets that have finite sub-Riemannian
	perimeter. We introduce a new notion of rectifiability that is,
	possibly, weaker than the one introduced by Franchi, Serapioni, and
	Serra Cassano. Namely, we consider subsets $\Gamma$ that, similarly to
	intrinsic Lipschitz graphs, have a cone property: there exists an open
	dilation-invariant subset $C$ whose translations by elements in
	$\Gamma$ don't intersect $\Gamma$. However, a priori the cone $C$ may
	not have any horizontal directions in its interior. In every Carnot
	group, we prove that the reduced boundary of every finite-perimeter
	subset  can be covered by countably many subsets that have such a cone
	property. The cones are related to the semigroups generated by the
	horizontal half-spaces determined by the normal directions.
	We further study the case when one can find horizontal directions
	in the interior of the cones, in which case we infer that
	finite-perimeter subsets are countably rectifiable with respect to
	intrinsic Lipschitz graphs. A sufficient condition for this to hold is the existence of a 
	horizontal one-parameter subgroup that is not an abnormal curve. As an
	application, we verify that this property holds in every filiform group, of either first or second kind.
\end{abstract}
\maketitle
\section{Introduction}
The celebrated rectifiability theorem by De Giorgi, see \cite{DeGiorgi54, DeGiorgi55}, states that the reduced boundary of a set of finite perimeter in the Euclidean space $\R^n$ is $C^1$-rectifiable, i.e., it can be covered, up to a negligible set with respect to the Hausdorff measure $\mathscr H^{n-1}$, by a countable union of $C^1$ hypersurfaces. The proof of this theorem relies on the fact that the blow-up of a set of finite perimeter at a point of its reduced boundary is a set with constant normal, and each constant-normal set in $\R^n$ is a half-space. 
The importance of having sufficiently regular sets of finite perimeter is evident on many classical key problems in Geometric Measure Theory and underlines the relevance of the notion of rectifiability in this context. A wide impact can be, for example, detected in developing a sufficiently rich theory for functions of bounded variation; see e.g.\ the monographs \cite{Federer, GiaquintaModicaSoucek, AFP, Maggi, EG}. 

In the more general context of metric measure spaces, 
the regularity of  finite-perimeter 
sets and the structure of some suitable notions of their boundaries
has been object of several studies in the last decades. We refer to this task as the \emph{rectifiability problem}.

In the current paper, we study the rectifiability problem in the setting of Carnot groups. A Carnot group $\G$ of step $s\in \N$ is a connected and simply connected Lie group whose Lie algebra $\g$ is stratified into $s$ layers, i.e., it is linearly decomposed as $\g=\g_1\oplus\dots\oplus\g_s$ with
\[
[\g_1,\g_i]=\g_{i+1}\text{ for }i=1,\dots,s-1, \quad \g_s\neq \{0\}\quad\text{and}\quad [\g_1,\g_s]=\{0\}.
\]
We refer the reader to \cite{FS, LeDonne} for an introduction to Carnot groups. 
Canonically, every Carnot group has a one-parameter family of dilations that we denote by $ \{\delta_t\,:\, t\geq 0\}$. We further fix a homogeneous distance $d$ on $\G$, which is unique up to biLipschitz equivalence. In this setting the notions of perimeter and constant-normal set can be naturally defined (see Definitions~\ref{def:finiteperimeter} and~\ref{def:constantnormal}). The original result of De Giorgi has been generalized by Franchi, Serapioni and Serra Cassano first in Heisenberg groups \cite{FSSCHeis} and then in all step-2 Carnot groups \cite{FSSC03}. The same result holds in the so-called type $\star$ Carnot groups, which are Carnot groups satisfying a suitable algebraic condition, see \cite{Marchi}, which generalizes the step-2 condition but may hold in arbitrary step. We also mention the recent \cite{LDM}, where the authors provide a class of Carnot groups that generalizes the type $\star$ class for which De Giorgi's rectifiability result holds. These are the only classes of Carnot groups in which the rectifiability problem has been solved in a satisfactory way, so far.
 
 In general, only few partial results are known.
In arbitrary groups, the first delicate issue concerns the blow-up analysis of the reduced boundary of a set of finite perimeter. By \cite{FSSCHeis} every tangent set at a point of the reduced boundary 
 is a set with constant horizontal normal, but this is not enough to prove that the tangent is a half-space, in general. Indeed, in the Engel group (the simplest Carnot group of step 3) there are examples of sets with constant horizontal normal that are not half-spaces; see \cite[Example 3.2]{FSSC03}. However, in the paper \cite{AKLD}, the authors prove that, among all the possible blow-ups of a set of constant horizontal normal in a Carnot group, there is always a half-space.  
 
 Another issue is to understand  what is the correct notion of rectifiable set  in  Carnot groups.
  Namely, at the moment it is not known which kind of rectifiability property one should expect for 
   finite-perimeter sets in these spaces.
  It is well-known that in the Euclidean setting, three equivalent notions of codimension-1 rectifiability are available: the countable covering family can be composed by $C^1$ hypersurfaces, Lipschitz images of sets in $\R^{n-1}$ or Lipschitz codimension-1 graphs; see e.g. \cite{Mattila75, Mattila95} for an account of rectifiability theory in the Euclidean spaces.  Actually, a very natural notion of rectifiability, via Lipschitz images of open subsets of Euclidean spaces, was given in the setting of metric spaces already by Federer in \cite{Federer} (and later by Ambrosio and Kirchheim in \cite{ AmbrosioKirchheimCurr}). Unfortunately, this notion does not fit in the geometric structure of a Carnot group since, according to this definition, already the Heisenberg group is purely unrectifiable (see \cite{AmbrosioKirchheimRect}). 
  
  Nonetheless, a definition of rectifiability using a suitable notion of Lipschitz graphs 
  or $C^1$ hypersurfaces can still be fruitful in the setting of Carnot groups. 
For this purpose, Franchi Serapioni and Serra Cassano introduced the notions of intrinsic Lipschitz graphs   (see Definition~\ref{def:introlipgraph}) and of  intrinsic $C^1$  hypersurfaces.
We know that the notion of rectifiability with respect to intrinsic $C^1$  hypersurfaces 
implies the one with intrinsic Lipschitz graphs; see e.g.\ \cite[Theorem 3.2]{V12}. 
This stronger notion of rectifiability   was used in \cite{ FSSCHeis, FSSC03, Marchi}. More precisely, the authors proved that  the reduced boundary of a set of finite perimeter in a type $\star$ group $\G$ can be covered, up to a set of $ \mathscr H^{Q-1} $-measure zero, by a countable union of intrinsic $C^1$  hypersurfaces, where $Q$ is the Hausdorff dimension of  $\G$.

It  is not known whether
the possibly weaker notion of rectifiability, obtained by replacing the intrinsic $C^1$ hypersurfaces by intrinsic Lipschitz graphs,  leads to an equivalent definition.
In fact,  the validity of an intrinsic Rademacher-type theorem is  still an unsolved problem.
One of the aim of this paper is to discuss another form of rectifiability that may be a priori even weaker than the intrinsic Lipschitz rectifiability. 

 The notion of intrinsic Lipschitz graph appeared in different equivalent forms in \cite{FSSC06, FSSC11, FS16} and we briefly recall it here (see also Section~\ref{sec:Lip} for a more complete discussion) in a way that is suitable for our purposes.
\begin{defn}\label{def:introlipgraph}
 Let $\G$ be a Carnot group and let  $ \mathbb W, \mathbb L \sus \G$ be homogeneous subgroups of $ \G $ such that $ \G = \mathbb W \cdot \mathbb L $ and $ \mathbb W \cap \mathbb L = \{0\} $. We say that $ \Sigma\sus \G $ is an (entire) \emph{intrinsic Lipschitz graph} if there exists $\beta>0$ such that
\begin{itemize}
\item[(i)] for every $p\in\Sigma$ one has 
\begin{equation}\label{eq:conoLip}
\Sigma\cap p\cdot \bigcup_{\ell\in \mathbb L\setminus \{0\}} B(\ell,\beta d(0,\ell))=\emptyset;
\end{equation}
\item[(ii)] $\pi_\mathbb W(\Sigma)=\mathbb W$.
\end{itemize}
 \end{defn}
The set $ \bigcup_{\ell\in \mathbb L} B(\ell,\beta d(0,\ell))$ is a homogeneous cone around $\mathbb L$ of aperture given by parameter $\beta$, while $\pi_{\mathbb W}$ is the canonical projection associated with the decomposition $\G=\mathbb W \cdot \mathbb L$. The cones introduced above are not the ones that are often used in the literature of Carnot groups, but they produce the same notion of intrinsic Lipschitz graph (see e.g. \cite[Remark A.2]{DMV})
 
 One of the goals of the current paper is to analyze the geometry of sets that possess a ``cone-behavior'' that resembles property \eqref{eq:conoLip} in the following weaker sense.

\begin{defn}
	A non-empty set $C\subseteq \G$ is said to be a \emph{cone} if 
	\[
	\delta_rC=C, \quad\text{ for any $r>0$.}
	\]
\end{defn}
 
 \begin{defn}\label{def:introconeprop}
 	Let $\G$ be a Carnot group. We say that a set $\Gamma\subseteq \G$ satisfies an \emph{(outer) cone property} if there exists a cone $ C \subseteq \G $ such that 
 	\[
 	\Gamma\cap pC=\emptyset, \quad\text{ for every $p\in \Gamma$.}
 	\]
 	
 \end{defn}
 
\noindent A dilation-invariant set with nonempty interior, which turns out to be related with sets of locally finite perimeter, is  the following semigroup. This notion has been introduced in \cite{BLD, BLD19} in the study of constant-normal sets in Carnot groups. (We point out that in \cite{BLD19} the authors are mainly interested in an \emph{inner cone property} of sets as opposed to the outer cone property.)

\begin{defn}\label{def:introcono}
 Let $\G$ be a Carnot group. For any $\nu\in \mathfrak g_1\setminus \{0\}$, the \emph{semigroup of horizontal normal $\nu$} is defined by
 \[
 S_\nu\coloneqq S(\exp(\nu^\perp+\R^+\nu)),
 \]
 where for any $A\subseteq \G$, the set $S(A)\coloneqq \bigcup_{k=1}^\infty A^k$ is the semigroup generated by $A$. 
 Here $\nu^\perp $ denotes the orthogonal space to $\nu$ within $\g_1$ with respect to some scalar product that we fixed on $\g_1$.
\end{defn}
 \noindent Semigroups with horizontal normal are cones and, by a standard argument of Geometric Control Theory (see \cite{AS}), they have non-empty interior. It can be also proved (see Proposition~\ref{prop:Snu:has:CN}) that any semigroup of horizontal direction $\nu$ has $\nu$ as constant horizontal normal. We first point out that semigroups of horizontal normal $\nu$ are minimal in the following sense: $S_\nu$ is contained in every set with $\nu$ as constant horizontal normal and for which the identity element 0 of $\G$ has positive density,     see \cite[Corollary 2.31 or Theorem 2.37]{BLD19} . This property, together with the dilation-invariance of such semigroups, allows us to perform a fruitful blow-up procedure and get our main result; see Theorem~\ref{cor:main}.

\begin{thm}\label{th:intromain}
Let $\G$ be a Carnot group and let $E\subseteq \G$ be a set of locally finite perimeter. Then  there exists a  family $\{C_h:h\in \mathbb N\}$ of open cones in $\G$ and a family $\{\Gamma_h:h\in \N\}$ of subsets of $\G$ such that each $\Gamma_h$ satisfies the $C_h$-cone property (as in Definition~\ref{def:introconeprop}) and 
\[
\mathcal F E=\bigcup_{h\in \N} \Gamma_h,
\]
where $\mathcal F E$ denotes the reduced boundary of $E$.
\end{thm}

Notice that the previous result is obtained without requiring any assumption on the Carnot group $\G$. The first natural question one may ask is when a set satisfying a cone property is also an intrinsic Lipschitz graph. We can point out a sufficient condition. Indeed, since each $C_h$ in the theorem above comes as a small shrinking of some semigroup $S_\nu$, it is enough that $S_\nu$ is such that 
there exists $X\in \g_1\setminus \{0\}$ with $\exp(X)\in \mathrm{int}(S_\nu)$ (see Remark~\ref{rmk:hyva} for more details). In Section~\ref{sec:Lip} we find some conditions on the group $\G$ that are sufficient to conclude that the reduced boundary of every set of finite perimeter in $\G$ is intrinsically Lipschitz rectifiable (see Definition~\ref{def:LR}). In particular, we notice that, whenever the group possesses a non-abnormal horizontal direction (see Definition~\ref{def:abnormal}), then each set $\Gamma_h$ appearing in Theorem~\ref{th:intromain} can be chosen to be an intrinsic Lipschitz graph (see Proposition~\ref{p:unicorn} and Corollary~\ref{c:horse}). The notions of normal and abnormal curve naturally appeared in Geometric Control Theory (see \cite{Montgomery}), and in the study of regularity of geodesics in sub-Riemannian manifolds (see  \cite{Vittone14}). In the paper \cite{LDMOPV}, it is proved that the one-parameter subgroup generated by $X\in \g_1$ is non-abnormal if and only if 
\begin{equation}\label{eq:algebraicintro}
\mathrm{span}\{\mathrm{ad}^k_X (\g_1) : k=0,\dots,s-1 \} = \g,
\end{equation}
where the adjoint of $X$ is defined by  $\mathrm{ad}_X^0(Y)\coloneqq Y$ and $\mathrm{ad}_X^k(Y)\coloneqq\mathrm{ad}_X^{k-1}([X,Y])$, for every $k\geq 1$ and every $Y\in \g$.
This gives us a purely algebraic sufficient condition on the group $\G$ for the intrinsic Lipschitz rectifiability of reduced boundaries of sets of finite perimeter. 
\begin{thm}\label{th:LRintro}
Let $\G$ be a Carnot group and assume there exists $X\in \g_1$ such that \eqref{eq:algebraicintro} holds. Then, the reduced boundary of every set of finite perimeter in $\G$ is intrinsically Lipschitz rectifiable. In particular, this result can be applied to all filiform Carnot groups.
\end{thm}
For the second part of Theorem~\ref{th:LRintro} see Section~\ref{Sec:Filiforms}. Among  filiform groups, we can find the Engel group, which is the simplest Carnot group of step $3$. We also point out that possessing a non-abnormal horizontal direction is stable under direct products and quotients. More precisely, if $\G_1$ and $\G_2$ are two Carnot groups that possess a non-abnormal horizontal direction, and $N$ is a normal subgroup of $\G_1$, then both $\G_1\times \G_2$ and $\G_1/N$ possess a non-abnormal horizontal direction (see Propositions~\ref{prop:products} and~\ref{prop:quotient}).

On the negative side, we point out that for example in the free group $\mathbb F_{2,3}$ of rank $2$ and step $3$, also known as Cartan group, all horizontal directions are abnormal and 
\[
\exp(\mathfrak f_1)\cap \mathrm{int}(S_\nu)=\emptyset, \quad\text{for every $\nu\in \mathfrak f_1\setminus \{0\},$}
\]
where $\mathfrak f_1$ denotes the horizontal layer of the Lie algebra of $\mathbb F_{2,3}$ (see \cite{BLD19}).

As an example of application, we remark that in \cite{DonVittone}, to study some properties of functions of bounded variation, the authors consider sub-Riemannian structures in which sets of finite perimeter have reduced boundary that is intrinsically Lipschitz rectifiable. 
 
The outline of the paper is the following. Section~\ref{sec:preliminaries} is devoted to the basic notions and facts related to Carnot groups. Section~\ref{sec:cones} is devoted to studying properties of closed semigroups of horizontal direction that will be crucial for the main theorem. Section~\ref{sec:main} contains the proof of Theorem~\ref{th:intromain}. Section~\ref{sec:Lip} contains the definitions of intrinsic Lipschitz graphs, intrinsic Lipschitz rectifiability and the proof of Theorem~\ref{th:LRintro} together with applications to some classes of Carnot groups. In Section~\ref{Sec:Filiforms} we introduce filiform groups and determine their abnormal lines and their graded Lie algebra automorphisms.

\vspace{0.5em}
\noindent{\bf Acknowledgements.} The authors warmly thank C.\ Bellettini, S.\ Rigot and the anonymous referee for precious comments and suggestions.

\section{Preliminaries}\label{sec:preliminaries} 
For an introduction to Carnot groups we refer the reader to \cite{FS, LeDonne}, while for a theory of sets of finite perimeter in Carnot groups, we refer to \cite{AKLD}.
In what follows, let $\G $ be a Carnot group of dimension $n$, let $\mathfrak g= \g_1\oplus \dots \oplus \g_s$ be its stratified Lie algebra and denote by $m\coloneqq \dim \g_1$ the rank of $\G$. We fix a scalar product $ \langle \cdot , \!\cdot \rangle $ on the horizontal layer $ \mathfrak g_1 $ of $ \g $ and a left-invariant Haar measure $\mu$ on $\G$. We endow $\G$ with the usual Carnot-Carath\'eodory metric $d$ and we denote by $B_r$ the metric ball of radius $r$ at the identity element of $\G$.

It is well-known that there exists a family $\{\delta_r\in \G^{\G}: r\geq 0\}$  such that $\delta_0\coloneqq 0$ and, for any $r>0$, $\delta_{r}$ is a graded diagonalizable automorphism of $\G$ satisfying the following properties. For any $r,s\geq 0$, $\delta_{rs}=\delta_r\circ\delta_s$ and, for every $x,y\in \G$ and $r\geq 0$
\[
d(\delta_r (x), \delta_r (y))=rd(x,y).
\] 
\begin{defn}\label{def:finiteperimeter}
Let $\Omega$ be an open set in $\G$. We say that a measurable set $E\subseteq \G$ has \emph{locally finite perimeter} in $\Omega$, if, for every $Y\in \mathfrak g_1$, there exists a Radon measure, denoted by $Y\ye$, on $\Omega$ such that
\[
\int_{A\cap E} Y\varphi \,d\mu=-\int_A \varphi \,d(Y\ye), \quad \text{for every open set $A\Subset \Omega$ and every $\varphi\in C_c^1(A).$}
\]
We say that $E$ has \emph{finite perimeter} in $\Omega$ if $E$ has locally finite perimeter in $\Omega$ and, for every basis $(X_1,\dots, X_m)$ of $\mathfrak g_1$, the total variation $|D\ye|(\Omega)$ of the measure $D\ye\coloneqq(X_1\ye,\dots,X_m\ye)$ is finite.
\end{defn}
\begin{defn}
 Let $ E \subseteq \G$ be a set with locally finite perimeter in $\G$. We define the \emph{reduced boundary} $ \mathcal{F}E $ of $ E $ to be the set of points $ p\in \G $ such that $ |\D\ye|(B(p,r))> 0 $ for all $ r > 0 $ and there exists
 \[
 \lim_{r \to 0}\dfrac{\D\ye (B(p,r))}{|\D\ye|(B(p,r))} \eqqcolon \nu_E(p)
 \]
 with $ |\nu_E(p)| = 1 $.
\end{defn}
\noindent We denote by $ \mathbb S(\G) $ the unit sphere of $\G$ and by $ \Sp(\g_1)$ the unit sphere in $ \mathfrak g_1 $.
\begin{defn}\label{def:constantnormal}
 Let $ E \subseteq \G $ be a set of locally finite perimeter in $\G$ and let $ \nu \in \mathbb S(\g_1)$. We say that $ E $ has $ \nu $ as \emph{constant horizontal normal} if $\nu \ye \geq 0$ in the sense of distributions and, for every $ X \in \nu^\perp\coloneqq\{ Y\in \mathfrak g_1: \langle \nu, Y\rangle=0\}$, one has $X \ye = 0$.
\end{defn}

\section{Cones in Carnot groups}\label{sec:cones}
In this section, we show that semigroups of horizontal normal $\nu$ represent the ``minimal'' sets having $\nu$ as constant horizontal normal (see Proposition~\ref{prop:Snu:has:CN} and Lemma~\ref{l:cn:contains:cone}).  Moreover, by Lemma~\ref{l:curves:Snu}, they are attainable sets of halfspaces in $\mathfrak g_1$. By a standard argument of Geometric Control Theory (see e.g.\ \cite[Theorem 8.1]{AS}), this implies that every $S_\nu$ has non-empty interior. 

An important result of this section is Proposition~\ref{prop:cone:continuity}, which proves that the semigroups $S_\nu$ are continuous with respect to the horizontal direction $\nu$. This fact, resumed in Remark~\ref{rmk:conecompactness}, will be used in the proof of Theorem~\ref{t:cone:property}.

The following proposition has been proven in \cite[Proposition 2.29]{BLD19}. We, however, write its short proof for the sake of completeness.

\begin{prop}
 \label{prop:Snu:has:CN}
 Let $\G$ be a Carnot group, let $ \nu \in \mathbb S(g_1) $ and let $ S_\nu $ be the semigroup with normal $ \nu $. Then $ S_\nu $ has $ \nu $ as constant horizontal normal.
\end{prop}
\begin{proof}
 By definition, $ \mathbb{1}_{S_\nu \cdot \exp(s X)} \leq \mathbb{1}_{S_\nu} $ for all $ s\geq 0 $ and $ X \in \{\nu\}\cup \nu^\perp $. Let $\varphi\in C_c^\infty(\G;[0,+\infty[)$ and denote by $ \Phi_X(p,s) \coloneqq p\exp(sX)$ the flow of $ X $ at time $ s $ starting from point $ p\in \G $. Then
 \begin{align*}
 \int_\G \mathbb{1}_{S_\nu} X\varphi \,\der \mu =&  \int_\G\mathbb{1}_{S_\nu} \cdot\lim_{s\to 0} \dfrac{1}{s}(\varphi(\Phi_X(\cdot, s))-\varphi)\,\der \mu \\
 =& \lim_{s\to 0}\dfrac{1}{s}\left(\int_\G \mathbb 1_{{S_\nu}\cdot \exp(sX)}\varphi\,\der \mu- \int_\G \mathbb{1}_{S_\nu} \varphi\,\der \mu  \right) \leq 0\,
 \end{align*}
 and therefore
 \[
 -\int_G \mathbb{1}_{S_\nu} X\varphi \,\der \mu \geq 0\,. 
 \]
 Hence  $ \langle X\mathbb{1}_{S_\nu}, \cdot \rangle $ is a positive linear functional and by Riesz's Theorem $ X\mathbb{1}_{S_\nu}$ is a Radon measure. Moreover, since for each  $X \in  \nu^\perp $ also $ -X \in  \nu^\perp $, we get that $ X \mathbb{1}_{S_\nu} = 0 $ for all $ X \in \nu^\perp $. Consequently, $ {S_\nu} $ has $ \nu $ as constant horizontal normal.
\end{proof}
\begin{lemma}
	\label{l:curves:Snu}
	Let $\G$ be a Carnot group, $ T> 0 $, and let $\gamma: [0,T]\to\G$ be a horizontal curve such that $ \gamma(0) = 0 $. If $ \langle \dot{\gamma}(t),\nu \rangle \geq 0 $ for almost every $ t  $, then $ \gamma(T) \in \overline{S}_\nu $.
\end{lemma}
\begin{proof}
	Fix $T>0$, define $X_1\coloneqq\nu$ and let $X_2,\dots,X_m$ be such that $(X_1, X_2, \dots, X_m)$ is an orthonormal basis for $\g_1$. We fix $T>0$. We can assume without loss of generality that $\dot \gamma \in L^\infty([0,T];\g_1)$. Let $u_1,\dots, u_m\in L^\infty([0,T])$ be such that
	\[
	\dot\gamma(t)=\sum_{i=1}^m u_i(t)X_i(\gamma(t)),
	\]
	for $\mathscr L^1$-almost every $t\in [0,T]$. Then, by assumption, we know that $u_1(t)\geq 0$ for $\mathscr L^1$-almost every $t\in [0,T]$. Consider piecewise constant sequences $(u_1^h),\dots, (u_m^h) $ in $L^\infty([0,T])$, with $h\in \N$, such that 
	\[
	u_i^h\rightarrow u_i \qquad \text{ in $L^1([0,T])$}, \quad \text{ as } h\to \infty,
	\]
	for any $i=1,\dots, m$ and such that $u_1^h(t)\geq 0$ for $\mathscr L^1$-almost every $t\in [0,T]$ 
	. We can also assume that $\sup_{h\in \mathbb N}\sum_{i=1}^m\|u_i^h\|_\infty\leq M$, for some $M>0$. According to to the definition of  $S_\nu$ and since $u_i^h$ are piecewise constant, the curves defined by
	\[
	\begin{cases}
	\dot\gamma^h(t)=\sum_{i=1}^m u_i^h(t)X_i(\gamma^h(t)) & \text{ for $\mathscr L^1$-a.e. $t\in [0,T]$,}\\
	\gamma^h(0)=0,
	\end{cases}
	\]
	are such that $\gamma^h(t)\in S_\nu$ for any $t\in [0,T]$. Since $d(\gamma^h(t),0)\leq Mt$ for every $t\in [0,T]$, there exists a compact set $K\subseteq \G$ for which 
	\[
	\bigcup_{h\in \mathbb N} \gamma^h([0,T])\cup \gamma([0,T])\subseteq K.
	\]
	We prove that 
	\[
	\lim_{h\to \infty} d(\gamma^h(T),\gamma(T))=0.
	\]
	It is not restrictive to work in coordinates and compute, for every $ t\in [0,T] $,
	\[
	\begin{aligned}
	|\gamma^h(t)-\gamma(t)|&=\left|\int_0^t\sum_{i=1}^m \left(u_i^h(\tau)X_i(\gamma^h(\tau))-u_i(\tau)X_i(\gamma(\tau))\right)\:d\tau\right|\\
	&\leq \int_0^t\sum_{i=1}^m |u_i^h(\tau)|\left|X_i(\gamma^h(\tau))-X_i(\gamma(\tau))\right|\:d\tau\\
	&\hphantom{\leq}+ \int_0^t\sum_{i=1}^m \left|u_i^h(\tau)-u_i(\tau)\right| |X_i(\gamma(\tau))|\:d\tau.
	\end{aligned}
	\]
	Notice that, by the choice of $u_i^h$, the term
	\[
	\alpha_h(t)\coloneqq\int_0^t\sum_{i=1}^m \left|u_i^h(\tau)-u_i(\tau)\right| |X_i(\gamma(\tau))|\:d\tau
	\]
	is infinitesimal as $h\to \infty$, and that, by the smoothness of $X_1,\dots,X_n$ and letting
	\[
	C \coloneqq \sup_{i=1,\dots,m} \mathrm{Lip}(X_i)(K) 
	\] and recalling the definition of $ M $, we have
	\[
	|\gamma^h(t)-\gamma(t)|\leq \alpha_h(t)+CM\int_0^t\left|\gamma^h(\tau)-\gamma(\tau)\right|\:d\tau,
	\]
	for all $ t\in [0,T] $.
	We are then in a position to apply Gr\"onwall Lemma to get 
	\[
	|\gamma^h(T)-\gamma(T)|\leq \alpha_h(T)e^{CMT},
	\]
	and letting $h\to \infty$ and by the arbitrariness of $ T $, we conclude the proof.
\end{proof}
\noindent Before proving Lemma~\ref{l:cn:contains:cone}, we make the following remark.
\begin{rmk}\label{rmk:dRdL}
 Let $ \G $ be a Carnot group and let $ d_L $ and $ d_R $ be the left- and right-invariant Carnot-Carathéodory distances on $ \G $ built with respect to the same scalar product on $\g_1$. Then the inversion is an isometry between $(\G,d_L)$ and $(\G,d_R)$. In particular, for every $p\in \G$ we have that $d_L(0,p)=d_L(p^{-1},0)=d_R(p,0)=d_R(0,p)$.
\end{rmk}
 Given $ p\in \G $ and a measurable set  $ F \subseteq \G$, we define the {\em lower and upper densities} $\theta_*(p,F)$ and $\theta^*(p,F)$ of $F$ at $p$ letting
\[
 \theta_*(p,F)\coloneqq \liminf_{r\to0}\frac{\mu(F\cap B(p,r))}{\mu(B(p,r))}\qquad\theta^*(p,F) \coloneqq \limsup_{r\to 0}\dfrac{\mu( F\cap B(p,r))}{\mu(B(p,r))}\,.
\]
The {\em measure theoretic boundary} $\partial^*F$ of $F$ is defined by
\[
\partial^*F\coloneqq\{p\in \G: \theta_*(p,F)>0 \text{ and } \theta^*(p,F)<1\}.
\]
The Lebesgue representative $\widetilde F$ of $F$ is defined by 
\begin{equation}
\label{Leb:rep}
\widetilde F\coloneqq \{p\in \G: \theta_*(p,F)=1\}.
\end{equation}
Notice that, by the Lebesgue's theorem, we know that $\widetilde F=F$ up to a set of $\mu$-measure zero. \\
The following proposition is proved in \cite[Proposition 3.6]{BLD19}
\begin{prop}\label{prop:enricoboundary}
	Let $F\subseteq \G$ be a set with $\nu$ as constant horizontal normal. Then 
	\[
	\mathcal F\widetilde F=\partial \widetilde F=\partial^*\widetilde F.
	\]
\end{prop}
\noindent Lemma~\ref{l:cn:contains:cone} below will be used in the proof of Lemma~\ref{l:measure:upperbound}.
\begin{lemma}
 \label{l:cn:contains:cone}
 Let $\G$ be a Carnot group, let $ F\subseteq \G $ be a set with $ \nu $ as constant horizontal normal and assume that $ \theta^*(0,F)> 0 $. Then $ \mathbb{1}_{S_\nu} \leq \mathbb{1}_F $ $ \mu $-almost everywhere.
\end{lemma}
\begin{proof}
	Recall first that every left-invariant Haar measure of a Carnot group is right-invariant, being Carnot groups nilpotent and therefore unimodular. \\ Since $F$ has  $ \nu $ as constant horizontal normal, by \cite[Lemma 3.1]{BLD} we have that, 
 \[
 \mathbb 1_{F\cdot p}\leq \mathbb 1_F\quad \text{for every $p\in S_\nu$,\quad $\mu$-almost everywhere.}
 \]
 
 If $B^L(p,r)$ and $B^R(p,r)$ denote, respectively, the metric balls built with respect to the left-invariant and right-invariant metrics, then for all $p\in S_\nu$
 \[
 \begin{aligned}
 \mu (F\cap B^L(0,r))\stackrel{\text{Rem.}\,\ref{rmk:dRdL}}{=}\mu( F\cap B^R(0,r)) &=\mu( F\cdot p \cap B^R(p,r))\\ &\leq \mu( F\cap B^R(p,r))\,.
 \end{aligned}
 \]
 Since $\theta^*(0,F)>0$, we deduce that for all $p\in S_\nu$,
 \begin{equation}\label{eq:density}
 \theta_{ F} ^R(p)\coloneqq \limsup_{r\to 0}\frac{\mu( F\cap B^R(p,r))}{\mu(B^R(p,r))}>0.
 \end{equation}
 On the other hand, by Lebesgue Theorem we have that $\theta_{F}^R(p)\in \{0,1\} $ for $\mu$-almost every $p\in \G$. By  \eqref{eq:density} we then get that $\mu$-almost every $p\in S_\nu$ is a Lebesgue point for $F$, hence for almost every $p\in S_\nu$ one also has $p\in F$, as required.
\end{proof}

In the following lemma, we denote by $\R^* O(n)$ the set of $n\times n$ matrices $A$ that can be written as
\[
A=\lambda B, 
\]
for some $\lambda\in\R\setminus \{0\}$ and some $B\in O(n)$.

\begin{lemma}\label{prop:opensets}
 Let $n\in \N$. Then the following facts hold.
 \begin{itemize}
  \item[(i)] Let $ H $ be a subgroup of $ GL(n) $. Then for every $ p\in \Rn $ and every open neighborhood $U$ of $ p $, there exists a neighborhood $M\subseteq H$ of the identity matrix such that
  \[
  \bigcap_{\ell\in M} \ell U
  \]
  is a neighborhood of $ p $.
  \item[(ii)] For every $p\in \R^n\setminus \{0\}$ and every open neighborhood $M\subseteq \R^* O(n)$ of the identity matrix, the set $Mp$ is an open neighborhood of $p$.
 \end{itemize}
\end{lemma}
\begin{proof}
 $(i)$ Consider $r, R\in (0,+\infty)$ and $ p\in \Rn $ such that $B(p,2r)\subseteq   U$ and $B(p,r)\subseteq B(0,R )$. 
We consider in $H$ the distance coming from the usual operator norm.
 We define 
 $M\coloneqq B(\mathrm{Id},r/R)^{-1}\subseteq H$. Then, take any $\ell\in M$, so that $\|\ell^{-1}-\mathrm{Id}\| <  r/R$. Then we have that for all $x\in B(0,R)$
 \[
 |\ell^{-1}(x)-x|\leq \|\ell^{-1}-\mathrm{Id}\|\, |x|\leq \frac{r}{R} R = r.
 \]
Consequently, by triangle inequality if $x\in B(p,r)$, and so also $x\in  B(0,R)$, we have that 
$$ |\ell^{-1}(x)-p|  \leq |\ell^{-1}(x)-x| + |x-p| \leq 2r.$$
Therefore, we showed that $\ell^{-1}(B(p,r))\subseteq B(p,2r)\subseteq U$, that is, $B(p,r)\subseteq \ell U$ for every $ \ell \in M $.
Hence, we infer that $B(p,r)\subseteq  \bigcap_{\ell\in M} \ell U$.
 
 $(ii)$ Consider the scaled orthogonal transformations $ \R^* O(n) $ acting continuously and transitively on $ \R^n\setminus\{0\} $. Fix $p\in \R^n\setminus \{0\}$ and let us denote by $(\R^* O(n))_p$ the stabilizer subgroup of $p$. By \cite[Theorem 3.2]{Helgason}, the mapping 
 \begin{align*}
 \psi \colon {\R^* O(n)}/{\left(\R^* O(n)\right)_p} &\to \R^n\setminus \{0\} \\
[\ell] \longmapsto&\ell(p) 	
 \end{align*}
 is a well-defined homeomorphism. Hence, since the projection $ \pi \colon \R^* O(n)\to  {\R^* O(n)}/{\left(\R^* O(n)\right)_p} $ is open, the map $\ell\mapsto \ell(p)$ obtained as $ \psi \circ \pi $ is open.
\end{proof} 
\begin{prop}\label{prop:cone:continuity}
 Let $\G$ be a Carnot group. Then, for every $\nu\in \mathbb S(\g_1)$, there exists an open neighborhood $U$ of $\nu$ in $\mathbb S (\g_1)$ such that 
 \[
 \bigcap_{\mu\in U} S_\mu
 \]
 has non-empty interior. Moreover, if $\nu \in \mathbb S(\g_1)$ and $X\in \g_1$ are such that $\exp(X)\in \mathrm{int}(S_\nu)$, then one can choose $U$ in such a way that $X\in \mathrm{int}( \bigcap_{\mu\in U} S_\mu)$.
\end{prop}
\begin{proof}
 Denote by $\mathbb F$ the free Carnot group of both the same rank and step of $\G$ (for a definition of free Lie algebra, see e.g.\ \cite[p.\ 45]{VSC} or \cite[p.\ 174]{Varadarajan}).
 Then, the canonical projection $\pi\colon\mathbb F\rightarrow \G$ induces a surjective isometry between the horizontal layers of the Lie algebras $\pi_*\colon\mathfrak f_1\rightarrow \g_1$\footnote{It is understood that the metric on $\g_1(\mathbb F)$ is the pull-back metric of the metric on $\g_1(\G)$}. Define $\widetilde\nu\in \mathfrak f_1$ such that $\pi_*(\widetilde\nu)=\nu$, and consider a non-empty open set $W\subseteq S_{\widetilde \nu}$, which exists by  \cite[Proposition 2.26]{BLD19}. The canonical action of $\R O(m)$ on $ \mathfrak f_1$ can be extended linearly to a map in $\mathrm{Aut}(\mathfrak f)$, since in free groups all non-horizontal left-invariant vector fields can be written in a (essentially) unique way as commutators of horizontal ones. We denote by $ H < \mathrm{Aut}(\mathfrak f) $ this group of automorphisms of $ \mathfrak f $ induced by the action of $\R O(m)$ on $ \mathfrak f_1$. Since automorphisms of $ \mathfrak f $ are linear bijections, we may interpret $ H $ as a subgroup of $ GL(n) $, where $ n $ is the topological dimension of $ \mathfrak f $. 
 
 Considering the corresponding Lie group automorphisms, with abuse of notation, we then identify the actions $\R O(m)\times \mathfrak f_1\rightarrow \mathfrak f_1$ and $\R O(m)\times \mathbb F\rightarrow \mathbb F$. Notice that these actions are continuous and open. By (i) of Lemma~\ref{prop:opensets} there exists an open neighborhood $M $ of the identity in $H$ such that
 \[
 \bigcap_{\ell\in M}\ell W
 \]
 has non-empty interior. Moreover, since $ H $ restricted to $ \mathfrak f_1 $ is isomorphic to $ \R^*O(m) $, by (ii) of Lemma~\ref{prop:opensets} the set $ M\widetilde{\nu}$ is an open neighborhood of $ \widetilde{\nu} $ in $ \mathfrak f_1 $. Set now $U\coloneqq\pi_*(M\widetilde\nu)$. Since $\pi_*$ is open, $U$ is open and for every $\mu \in U$ we can find $m_0\in M$ such that $\mu=\pi_*(m_0\widetilde \nu)$. By properties of orthogonal matrices, notice that we also have $\pi_*((m_0\widetilde \nu)^\perp)=\mu^\perp$. Since $ \pi $ is open as well, the inclusions
 \begin{equation}\label{eq:inclusioni}
 \begin{aligned}
 \pi\left(\mathrm{int}\bigcap_{\ell\in M}\ell W\right)&\subseteq \pi\left(m_0W\right)\subseteq \pi \left(m_0 S_{\widetilde\nu}\right)=\pi(m_0(S(\exp(\widetilde\nu^\perp+\R^+\widetilde \nu)))) \\
 &=\pi S(\exp((m_0\widetilde \nu)^\perp+\R^+m_0\widetilde \nu))=S(\exp(\mu^\perp+\R^+\mu))=S_\mu
 \end{aligned}
 \end{equation}
 conclude the first part of the proof.
 
 Assume now that $\exp(X)\in \mathrm{int}(S_\nu)$. Then, choosing $ \widetilde{X} \in \mathfrak f_1$  and an open set $ W \subseteq S_{\widetilde{\nu}} $  such that $ \pi_*(\widetilde{X}) = X $ and $ \exp(\widetilde{X}) \in W $, one can repeat the argument of the previous part of the proof with the additional condition that, by (i) of Lemma~\ref{prop:opensets},
 \[
 \exp(\widetilde X)\in\mathrm{int}\left (\bigcap_{\ell\in M}\ell W\right ).
 \]
 The proof is again finished by \eqref{eq:inclusioni}.
\end{proof}
For any subset $ A \subseteq \Sp(\G) $, we define
\begin{equation}\label{conogelato}
C(A) \coloneqq \{ \delta_r(a) : a \in A,\, r> 0 \}.
\end{equation} 
Notice that $C(A)$ is a cone.
\begin{rmk}\label{rmk:conecompactness}
 As an immediate consequence of Proposition~\ref{prop:cone:continuity}, we notice that there exist $N\in\N$ and a finite family of triples $(U_i,\Omega_i,K_i)_{i=1}^N$ such that
 \begin{enumerate}
  \item $(U_i)_{i=1}^N$ is an open covering of $\Sp(\g_1)$;
  \item for every $i=1,\dots,N$, $\Omega_i$ is a nonempty open subset of $\Sp(\G)$ such that
  \[
  C(\Omega_i)\subseteq\bigcap_{\nu\in U_i} S_\nu;
  \]
  \item for every $i=1,\dots,N$, $K_i\subseteq \Omega_i$ is a compact set with nonempty interior in $\Sp(\G)$.
\end{enumerate} 
  Moreover, if for every $\nu\in \mathbb S(\g_1)$, there exists $X\in \g_1$ such that $\exp(X)\in \mathrm{int}(S_\nu)$, then each compact set $K_i$ can be chosen in such a way that there exists $X_i\in \g_1$ with $\exp(X_i)\in \mathrm{int}(C(K_i))$.
 
\end{rmk}
\section{Proof of Theorem 1.5}\label{sec:main}
In this section we show that, in Carnot groups, the reduced boundary of sets of locally finite perimeter can be decomposed into countably many pieces satisfying a cone property. This is precisely stated in Theorem~\ref{cor:main}, whose proof and statement complete and combine the previous Lemmata~\ref{l:measure:upperbound},~\ref{positivemeasure}, and~\ref{t:cone:property}. 
\begin{rmk}
	\label{rmk:C_inv_cone_property}
	If a subset $\Gamma$ of a Carnot group  satisfies the $C$-cone property, for some cone $C$, then $\Gamma$ also satisfies the $C^{-1}$-cone property.
Indeed,
	 assume by contradiction that there exist $p\in \Gamma$ and $q\in \Gamma\cap pC^{-1}$. Since $q\in pC^{-1}$, then $p\in qC$. But $\Gamma $ satisfies a $C$-cone property and, since $q\in \Gamma$ we have 
	$
	\Gamma \cap qC=\emptyset.
	$
	This is in contradiction with $p\in \Gamma\cap qC$.
\end{rmk}

\begin{lemma}
 \label{l:measure:upperbound}
 Let $\G$ be a Carnot group. If $ E \subseteq  \G$ is a set of locally finite perimeter in $\G$ and  $ \varepsilon > 0 $, 
 then, for every $ p \in \mathcal{F}E $, 
 there exists $ r_p > 0 $ such that for any $ r \in (0,r_p) $ we have
 \begin{equation*}
 \mu(B_r\cap  S_{\nu_E(p)}\cap  p^{-1}(\mathbb G\setminus E))\leq \varepsilon r^Q
 \end{equation*}
 and 
 \begin{equation*}
 \mu(B_r\cap  S_{\nu_E(p)}^{-1}\cap p^{-1}E)\leq \varepsilon r^Q.
 \end{equation*}
\end{lemma}
\begin{proof}
  Fix  $p\in \mathcal FE$ and notice that $E$ admits a tangent at $p$ that has, by the same argument of \cite[Theorem 3.1]{FSSC03}, $\nu_E(p)$ as constant horizontal normal. This means that, for any sequence $r_h\to 0$, we can find a set $F$ with constant horizontal normal $\nu_F\equiv \nu_E(p)$ and a subsequence $r_{h_k}$ such that
  \[
  \mathbb 1_{\delta_{r^{-1}_{h_k}}(p^{-1}E)}\rightarrow \mathbb 1_F \text{ in $L^1_{\rm loc}(\G)$}, \qquad \text{as }k\to+\infty.
  \]
  It is not restrictive to assume that $\widetilde F=F$, where $\widetilde F$ is as in \eqref{Leb:rep}. By Proposition~\ref{prop:enricoboundary}, since $0\in \mathcal F \widetilde F$ then $0\in \partial^* \widetilde F$ and by Lemma~\ref{l:cn:contains:cone}, one has
  \[
 \mathbb 1_{\delta_{r^{-1}_{h_k}}(p^{-1}E)\cap S_{\nu_E(p)}} \rightarrow \mathbb 1_{S_{\nu_E(p)}} \text{ in $L^1_{\rm loc}(\G)$}, \qquad \text{as }k\to+\infty.
 \]
Consequently, we have          \[
 \mathbb 1_{\delta_{r^{-1}}(p^{-1}E)\cap S_{\nu_E(p)}} \rightarrow \mathbb 1_{S_{\nu_E(p)}} \text{ in $L^1_{\rm loc}(\G)$},\qquad \text{as }r\to0.
 \]
   This completes the proof of the first statement. The proof of the second inequality follows from the first one by replacing $E$ with $\G\setminus E$ and recalling that $S_{-\nu}=S_{\nu}^{-1}$.
\end{proof}
\noindent
Regarding next lemma, recall the notation $ C(\Omega) $ introduced in \eqref{conogelato}.
\begin{lemma}\label{positivemeasure}
 Let $\G$ be a Carnot group, let $K \subseteq  \Omega \subseteq \Sp(\G)  $ be such that $ K $ is compact and $ \Omega $ is open.
Then
 \[
 \eta(C(\Omega), C(K))\coloneqq\inf \left\{\mu(C(\Omega)\cap B_2\cap \xi C(\Omega)^{-1} \cap \xi B_2): \xi \in K\right\}> 0.
 \]
\end{lemma}
\begin{proof}
 Since $K $ is a compact set, there exists $\delta\in (0,1)$ such that
 \[
 0 < \delta < d(K, \mathbb G\setminus C(\Omega)).
 \]
 Hence for any $\xi \in K$ one has $\xi B_{ \delta}\subseteq C(\Omega)\cap B_2$. Therefore
 \[
 \begin{aligned}
 \mu(C(\Omega)\cap B_2\cap \xi C(\Omega)^{-1}\cap \xi B_2)&\geq \mu (C(\Omega)\cap B_2\cap \xi C(\Omega)^{-1}\cap \xi B_{ \delta})\\
 &=\mu (\xi C(\Omega)^{-1}\cap \xi B_{ \delta}) \\
 &=\mu (C(\Omega)^{-1}\cap  B_{\delta}),
 \end{aligned}
 \]
 which is a positive lower bound independent of $ \xi \in K $.
\end{proof}
\noindent By combining Remark~\ref{rmk:conecompactness} and Lemmata~\ref{l:measure:upperbound} and~\ref{positivemeasure} we get the following corollary.
\begin{corollary}\label{prop:prop2}
Let $\G$ be a Carnot group, let $(U_i,\Omega_i, K_i)_{i=1}^N$ be as in Remark~\ref{rmk:conecompactness}, and let $E\subseteq\G$ be a set of locally finite perimeter. Then, for every $p\in\mathcal F E$, there exists $r_p>0$ such that, if $\nu_E(p)\in U_i$ and $r\in(0,r_p)$, then 
\begin{equation}\label{inequality1}
\mu (B_{2r}\cap C_i\cap  p^{-1}(\G\setminus E))\leq \dfrac{\eta_i}{3} r^Q
\end{equation}
and 
\begin{equation}\label{inequality2}
\mu (B_{2r}\cap C_i^{-1}\cap p^{-1}E)\leq  \dfrac{\eta_i}{3} r^Q,
\end{equation}
where $C_i\coloneqq C(\Omega_i)$ and $ \eta_i\coloneqq\eta(C_i,C(K_i))= \inf\{ \mu (C_i\cap B_2  \cap \xi \inv{C}_i\cap \xi B_2) \,:\, \xi \in K_i\}$.
\end{corollary}
\begin{lemma}
 \label{t:cone:property}
Let $\G$ be a Carnot group and let $E$ be a set of locally finite perimeter in $\G$. Consider a family $(U_i,\Omega_i,K_i)_{i=1}^N$ as in Remark~\ref{rmk:conecompactness} and, for every $p\in\mathcal F E$, define $r_p>0$ as in Corollary~\ref{prop:prop2}. Then, the sets
%
%
\begin{equation}\label{eq:effeielle}
F_{i,\ell} \coloneqq \{ p\in \mathcal{F}E \,:\, \nu_E(p)\in U_i, r_p > \tfrac{1}{\ell} \},
\end{equation}
defined for $i\in\{1,\dots,N\}$ and $\ell\in\N$, satisfy
\begin{equation}\label{puppa}
 F_{i,\ell}\cap pB_{1/\ell} \cap pC(K_i)=\emptyset \qquad\forall\ p\in F_{i,\ell}.
\end{equation}
\end{lemma}
\begin{proof}
Denote for shortness $\widetilde C_i\coloneqq C(K_i)$. Suppose by contradiction that there exist $x\in F_{i,\ell}$ and $y \in F_{i,\ell}\cap xB_{1/\ell}\cap x\widetilde C_i$ with $y\neq x$. Let $\eta_i$ be as in Corollary~\ref{prop:prop2} and define 
\[
r\coloneqq d(x,y) \quad\text{and}\quad I\coloneqq xC_i\cap yC_i^{-1}\cap xB_{2r}\cap yB_{2r}.
\]
By construction, we have that
\[
r=d(x,y)<\tfrac 1\ell<\min\{r_x,r_y\}
\]
 and using the facts that $I\subseteq xC_i \cap xB_{2r}$ and $I\subseteq yC_i^{-1} \cap yB_{2r}$, by applying \eqref{inequality1} and \eqref{inequality2} we have
 \[
 \mu (I\setminus E)\leq\mu(xC_i\cap xB_{2r} \cap (\G\setminus E))\leq \tfrac{\eta_i}{3}r^Q
 \]
 and
 \[
 \mu(I\cap E) \leq \mu(yC_i^{-1}\cap yB_{2r}\cap E)\leq\tfrac{\eta_i}{3}r^Q.
 \]
  This contradicts the fact that  by Lemma~\ref{positivemeasure}, we have $\mu(I)\geq \eta_i r^Q$. 
\end{proof}

\begin{thm}\label{cor:main}
Let $\G$ be a Carnot group and let $E\subseteq \G$ be a set of locally finite perimeter in $\G$. Then there exist a countable family $\{C_h:h\in \mathbb N\}$ of open cones in $\G$ and a countable family $\{\Gamma_h:h\in \N\}$ of subsets of $\G$ 
such that each $\Gamma_h$
satisfies the $C_h$-cone property and
\[
\mathcal F E= \bigcup_{h\in\N}\Gamma_h.
\]	
\end{thm}
\begin{proof}
	It is enough to show that, for any $i,\ell\in\N$, the set $F_{i,\ell}$ defined in \eqref{eq:effeielle} can be covered by a countable union of sets satisfying a cone property. This is simply done by covering $F_{i,\ell}$ by a countable family of balls $\{B(p_j,1/\ell):p_j\in F_{i,\ell}, j\in\N\}$ and decomposing $F_{i,\ell}=\bigcup_{j\in\N}(F_{i,\ell}\cap B(p_j,1/\ell))$. We now set 
	\begin{equation}\label{eq:gammailj}
	\Gamma_{i,\ell}^j\coloneqq F_{i,\ell}\cap B(p_j,1/\ell)
	\end{equation}
	Then by \eqref{puppa}, we have that, for every $i,j\in \mathbb N$, the set $\Gamma_{i,\ell}^j$ has the $C(\mathrm{int}(K_i))$-cone property, where $K_i$ are the compact subsets of the sphere introduced in Remark \ref{rmk:conecompactness}. Up to relabeling the family $\{\Gamma_{i,\ell}^j: i=1,\dots,N, \ell,j\in \mathbb N\}$ and renaming $C_i\coloneqq C(\mathrm{int}(K_i))$, we have then a countable family $\{C_h: h\in \mathbb N\}$ of open cones  and a countable family $\{\Gamma_h:h\in \mathbb N\}$ of sets such that each $\Gamma_h$ satisfies the $C_h$-cone property and $\mathcal F{E}=\bigcup_{h\in \mathbb N}\Gamma_h$.
\end{proof}

\begin{rmk}
	Notice that the family of cones $\{C_h:h \in \mathbb N\}$ appearing in Theorem \ref{cor:main} is indeed finite. This comes from the construction $C_i=C(\mathrm{int}(K_i))$, and the fact that, by Remark \ref{rmk:conecompactness}, the family $\{K_i: i=1,\dots, N\}$ is finite.
	\end{rmk}
\section{Intrinsic Lipschitz graphs}\label{sec:Lip}
In this section we follow \cite{FSSC06} to recall the notion of intrinsic Lipschitz graph in Carnot groups. The construction of an intrinsic Lipschitz graph requires the space to have a decomposition into complementary subgroups (see Definition~\ref{def:introlipgraph}). As observed in Remark~\ref{rmk:hyva} below, in certain cases Theorem~\ref{cor:main} can be strengthened to deduce that the reduced boundary of sets of finite perimeter is intrinsically Lipschitz-rectifiable, i.e., it can be covered by a countable union of intrinsic Lipschitz graphs. \\
 Since we only deal with the notion of codimension-one rectifiable sets, Definition~\ref{def:introlipgraph} will be used only in case $\dim \mathbb L=1$. In this situation, according to \cite[Proposition 3.4]{V12} (see also \cite{FS16}), one can see that, up to a modification of the parameter $\beta>0$, sets $\Sigma$ satisfying point (i) of Definition~\ref{def:introlipgraph} can be extended to sets $\Sigma\subseteq \widetilde \Sigma$ that satisfy (ii).  
\begin{rmk}\label{rmk:hyva}
	Assume $\Gamma\subseteq \G$ is a set with the $C$-cone property such that $C$ is open and there exists $X\in \g_1\setminus\{0\}$ with the property that $\exp(X)\in C$. Then $\Gamma$ is an intrinsic Lipschitz graph. Indeed, according to Definition~\ref{def:introlipgraph}, we may choose $ \mathbb{L} \coloneqq \{\exp(tX): t \in \R \} $ and $ \mathbb{W} \coloneqq \exp(X^\perp \oplus \g_2 \oplus \cdots \oplus \g_s)  $. Recall that, by Remark~\ref{rmk:C_inv_cone_property}, the set $ \Gamma $ satisfies also $ C^{-1} $-cone property. Since both $C$ and $ C^{-1} $ are open, there exists $\varepsilon>0$ such that $B(\exp(X),\varepsilon)\subseteq C$ and $B(\exp(-X),\varepsilon)\subseteq C^{-1}$. By scaling we may assume that $ \|X\| = 1 $, and so with the choice $\beta=\varepsilon$ condition \eqref{eq:conoLip} is satisfied.
	
In particular, Lemma~\ref{t:cone:property} shows that, if for some $i=1,\dots, N$, the set $K_i$ defined in Remark~\ref{rmk:conecompactness} is such that $ \Int(K_i)\cap\exp(\g_1)\neq \emptyset$, then, for every $j,\ell\in \mathbb N$, the set $\Gamma_{i,\ell}^j$ defined in \eqref{eq:gammailj} is an intrinsic Lipschitz graph. 
\end{rmk}
\begin{defn}\label{def:LR}
 Let $ \G $ be a Carnot group of homogeneous dimension $Q$ and let $ E \sus \G $. We say that $ E $ is \emph{ intrinsically Lipschitz rectifiable }if there exists a countable family $ \{\Sigma_h : h\in \N\} $ of intrinsic Lipschitz graphs such that
 \[
 \mathscr H^{Q-1}\left (E \setminus \bigcup_{h \in \N}\Sigma_h\right ) = 0.
 \]
\end{defn}
\noindent An immediate consequence of Remark~\ref{rmk:hyva} and Theorem~\ref{cor:main} is given by Corollary \ref{p:unicorn} below.
\begin{corollary}\label{p:unicorn}
Let $\G$ be a Carnot group and assume that for all $ \nu \in \Sp(\g_1) $ there exists $ X \in \Sp(\g_1) $ such that $ \exp(X) \in \mathrm{int}( S_\nu) $. Then the reduced boundary of every set of locally finite perimeter in $\G$ is intrinsically Lipschitz rectifiable.
\end{corollary}
To describe some conditions on the group that guarantee the validity of the assumptions of Corollary~\ref{p:unicorn}, we introduce the definition of end-point map. 
\begin{defn}\label{def:endpoint}
 Let $\G$ be a Carnot group. The \emph{end-point map} $\mathrm{End}\colon L^\infty([0,1];\g_1)\to \G$ is defined by letting
 \[
 \mathrm{End}(h)=\gamma(1),
 \]
 where $\gamma:[0,1]\to\G$ is the horizontal curve that is the unique solution of
 \[
 \begin{cases}
 \dot\gamma(t)=h(t) \\
 \gamma(0)=0.
 \end{cases}
 \]
 With abuse of notation we also write $\mathrm{End}(\gamma)$ meaning $\mathrm{End}(h)$ for the defining control $h$. 
\end{defn}
In what follows, we say that a map $ F \colon M \to N $ between topological spaces $ M $ and $ N $ is \emph{locally open at} $ p \in M $ if, for every neighborhood $ U $ of $ p $, the set $ F(U) $ is a neighborhood of $ F(p) $. 
If $\G$ is a Carnot group and $X\in \g_1$ is a horizontal direction, we also say that the end-point map $\mathrm{End}\colon L^\infty([0,1];\g_1)\to \G$ is locally open at $X$, if it is locally open at $h(t)\equiv X$.\\
Before proving Lemma~\ref{l:elk}, we point out some topological properties of the semigroups $S_\nu$.
\begin{lemma}\label{prop:topology}
	Let $\G$ be a Carnot group and let $\nu\in\g_1\setminus\{0\}$. Then $\mathrm{int}(S_\nu)=\mathrm{int}(\overline{S}_\nu)$. 
\end{lemma}
\begin{proof}
	Since $\mathrm{int}(S_\nu)$ has the (inner) $\mathrm{int}(S_\nu)$-cone property, then by \cite[Lemma 2.36]{BLD19} the set $ \mathrm{int}(S_\nu) $ is regularly open, i.e., we have that
	\begin{equation}\label{eq:regularlyopen}
	\mathrm{int}(S_\nu)=\mathrm{int}(\overline{\mathrm{int}(S_\nu)}).
	\end{equation}
	On the other hand, by \cite[Theorem 8.1]{AS}, we also have that
	\begin{equation}\label{eq:regularlyclosed}
	\overline{S}_\nu=\overline{\mathrm{int}(S_\nu)}.
	\end{equation}
	The result follows combining \eqref{eq:regularlyopen} and \eqref{eq:regularlyclosed}. 
\end{proof}

\begin{lemma}\label{l:elk}
 Let $\G$ be a Carnot group, let $X\in \g_1\setminus\{0\}$ and assume that the end-point map $ \mathrm{End} \colon L^\infty( [0,1];\g_1)\to \G $ is locally open at $ X $.  Then, for every $ \nu \in \Sp(\g_1) $ satisfying $ \langle \nu,X \rangle > 0$, we have $ \exp(X) \in \mathrm{int} (S_\nu)$.
\end{lemma}
\begin{proof}
Let $\varepsilon\coloneqq \langle \nu, X\rangle >0$ and
 \[
 B_\epsilon(X) \coloneqq \{v \in L^\infty([0,1];\g_1) : \| X-v\|_\infty < \varepsilon \}.
 \]
 Since $ \mathrm{End} $ is open at $ X $, by Lemma~\ref{prop:topology} it suffices to show that $ \mathrm{End}(B_\varepsilon(X)) \sus \overline{S}_\nu $. On the other hand,  by Lemma~\ref{l:curves:Snu}, if  $ v \in L^\infty([0,1];\mathfrak g_1)$ satisfies $ \langle v(t),\nu \rangle > 0 $ for almost every $ t\in [0,1] $, then $\mathrm{End}(v)\in  \overline{S}_\nu$. The proof is then achieved by noticing that, for every $ v \in B_\varepsilon(X)  $, we have
 \[
 \langle v(t),\nu \rangle = \langle \nu, X \rangle -\langle \nu, X-v(t) \rangle \geq \varepsilon- \| \nu\| \| X-v\|_{\infty} > 0
 \]
 for almost every $ t \in [0,1]$.
\end{proof}
\begin{corollary}\label{c:horse}
 Let $\G$ be a Carnot group and assume there exists a basis  $ \{X_i : i =  1,\dots, m \} $ of $ \g_1 $ such that the end-point map $ \mathrm{End} \colon L^\infty( [0,1];\g_1)\to \G $ is locally open at $ X_i $, for every $i=1,\dots,m$. Then the reduced boundary of every set of locally finite perimeter in $\G$ is intrinsically Lipschitz rectifiable.
\end{corollary}
\begin{proof}
 It is enough to combine Corollary~\ref{p:unicorn}, Lemma~\ref{l:elk} and the following fact. If $X\in \g_1$ and $ \exp(X) \in \mathrm{int}(S_\nu) $, then $ \exp(-X)\in \mathrm{int}(S_{-\nu})$.
\end{proof}
\begin{rmk}
	Every Carnot group possessing a spanning set of pliable or strongly pliable vectors in the sense of \cite{JS17} and \cite{SS18}, respectively, has the property that the reduced boundary of any set of locally finite perimeter is intrinsically Lipschitz rectifiable.
\end{rmk}
\noindent
We next give a sufficient condition that has an equivalent algebraic formulation, and that can be more easily verified. This will be used to deduce that, for example, Corollary~\ref{c:horse} applies to \emph{filiform groups} (see Section \ref{Sec:Filiforms}).
\begin{defn}\label{def:abnormal}
  Let $\G$ be a Carnot group and let $\gamma:[0,1]\to\G$ be a horizontal curve. We say that $\gamma$ is \emph{non-abnormal} if $\der\mathrm{End}(\gamma)$ has full rank. We also say that a horizontal direction $X\in \g_1\setminus\{0\}$ is non-abnormal, if $t\mapsto\exp(tX)$ is non-abnormal.
\end{defn}
As pointed out in \cite{ABB} and in \cite{Montgomery}, the fact that a curve is abnormal does not depend on its parametrization and one can develop a theory considering the end-point map defined on any $L^p$ space, $1\leq p\leq \infty$. Moreover, the Volterra expansion (see e.g.\ \cite[Formula (6.9)]{ABB}), allows to compute the differential of the End-point map with respect to any variation in $L^p$. In particular, if the differential of $\mathrm{End}\colon L^2([0,T];\R^m)\rightarrow \G$ has full-rank at $X\in \g_1$, then also the differential of $\mathrm{End}\colon L^\infty([0,T];\R^m)\rightarrow \G$ has full rank at $X$. This observation allows us to consider the formula for the differential of the end-point map in $L^2$ developed by \cite{LDMOPV}.
\begin{prop}\label{prop:nonabnormal}
 Let $\G$ be a Carnot group and let $ X \in \g_1\setminus\{0\}$. Then the curve $ t\mapsto \exp(t X) $ is non-abnormal if and only if
 \begin{equation}\label{eq:adjoint}
 \mathrm{span}\{\mathrm{ad}^k_X (\g_1) : k =0,\ldots,s-1 \} = \g.
 \end{equation}
\end{prop}
\begin{proof}
 Denote by $\gamma(t)\coloneqq \exp(tX)$. It is enough to notice that, by \cite[Proposition 2.3]{LDMOPV}
 \[
 \mathrm{Im}(\der\mathrm{End}(\gamma))= \der R_{\gamma(1)}\mathrm{span}\{\mathrm{Ad}_{\gamma(t)}\g_1: t\in (0,1)\}=\der R_{\gamma(1)}\mathrm{span}\{\mathrm{ad}^k_X \g_1: k\in \N\}. \qedhere
 \]
\end{proof}
\begin{rmk}
 \label{rmk:U:notabnormal}
 Condition~\eqref{eq:adjoint} is clearly open in $X$. In particular, if $\G $ contains a non-abnormal curve, then there exists an open set $ U \sus \Sp(\g_1) $ such that $ U = -U $ and $ t\mapsto \exp(t X') $ is non-abnormal for all $ X' \in U $.
\end{rmk}
\begin{corollary}\label{cor:LR}
 Let  $ \G $  be a Carnot group and assume there exists a non-abnormal $ X\in \g_1 \setminus\{0\}$. Then the reduced boundary of any set of locally finite perimeter in $\G$ is intrinsically Lipschitz rectifiable.
\end{corollary}
\begin{proof}
 Since $\gamma(t)\coloneqq\exp(tX)$ is non-abnormal, then $\der\mathrm{End}(\gamma)$ has full rank and, in particular, $\mathrm{End}$ is locally open at $X$. The proof then follows by combining Corollary~\ref{c:horse} and Remark~\ref{rmk:U:notabnormal}.
\end{proof}
%
\begin{prop}\label{prop:products}
 Let $ \G_1 $ and $ \G_2 $ be two Carnot groups possessing non-zero non-abnormal horizontal directions. Then the Carnot group $ \G_1 \times \G_2 $ possesses a non-zero non-abnormal horizontal direction.
\end{prop}
\begin{proof}		
  Denote by $\g(\G_i)$ the Lie algebra of $\G_i$ and by $\g_1(\G_i)$ the related horizontal layer, for $i=1,2$.
  Recall that the Lie bracket of the product algebra $\g(\G_1)\times \g(\G_2)$ is defined by
  \[
  [(Y_1,Y_2), (Z_1,Z_2)]=([Y_1, Z_1],[Y_2,Z_2]),
  \]
  for every $Y_1,Z_1\in \g(\G_1)$ and every $Y_2,Z_2\in \g(\G_2)$. Then, by induction on $k$, one can check that
  \[
  \mathrm{ad}^k_{(Y_1,Y_2)}(Z_1,Z_2)= (\mathrm{ad}^k_{Y_1}(Z_1),\mathrm{ad}^k_{Y_2}(Z_2)),
  \]
  for every $k\in \mathbb N$, every $Y_1,Z_1\in \g(\G_1)$ and every $Y_2,Z_2\in \g(\G_2)$.
  
   Let $X_1\in \g_1(\G_1)$ and $X_2\in \g_1(\G_2)$ be non-zero non-abnormal directions for $\G_1$ and $\G_2$, respectively. Then $(X_1,X_2)$ is non-abnormal for $\G_1\times \G_2$. To prove this it is enough to notice that for any $k\in \N$ one has
  \[
  \mathrm{ad}^k_{(X_1,X_2)}(\g_1(\G_1)\times\g_1(\G_2))=\mathrm{ad}^k_{X_1}(\g_1(\G_1))\times\mathrm{ad}^k_{X_2}(\g_1(\G_2))\qedhere
  \]
\end{proof}
\begin{prop}\label{prop:quotient}
 Let $ \G $ be a Carnot group possessing a non-zero non-abnormal horizontal direction and assume that $ N \trianglelefteq \,\G$ is a normal subgroup of $\G$. Then, the Carnot group $ \G /N $ possesses a non-zero non-abnormal horizontal direction.
\end{prop}
\begin{proof}
  Let $X\in \g_1\setminus\{0\}$ be a non-abnormal direction for $\G$. Then, if $\pi:\G\rightarrow \G/N$ is the canonical projection, the push-forward vector $\pi_* X$ is non-abnormal for $\G/N$. This is true by the fact that
  \[
\pi\mathfrak g=\mathrm{span}\{\pi_* \mathrm{ad}^k_X \g_1:k\in \N\}=\mathrm{span}\{\mathrm{ad}^k_{\pi_* X} \pi_* \g_1: k\in \N\},
\]
and by noticing that $\pi_* \g_1=\g_1(\G/N)$. If $\pi_*X\neq 0$, the proof is concluded. If $\pi_*X=0$, then $0\in~\g_1(\G/N)$ is non-abnormal and, since non-abnormality is an open and dilation-invariant condition, all the elements in $\g_1(\G/N)$ are non-abnormal.
\end{proof}
\section{Filiform groups}
\label{Sec:Filiforms}
\noindent
In this section we briefly introduce filiform groups and study their abnormal horizontal lines and automorphisms. As a corollary, we obtain that, in all filiform groups, the reduced boundary of any set of locally finite perimeter is intrinsically Lipschitz rectifiable.
\begin{defn}
	We say that a Carnot group $\G$ is a \emph{filiform group} of step $s$ if the stratification $\g=\g_1\oplus\dots\oplus\g_s$ of its Lie algebra satisfies $\dim \g_1=2$ and $\dim \g_i=1$, for every $i=2,\dots,s$.
\end{defn}
\begin{defn}\label{def:filiforms_first_and_second}
	Let $ \G $ be a filiform group of step $ s $.
	If its Lie algebra $ \g $ has a basis $ \{X_0,\dots,X_s \} $ with the only nonzero bracket relations
	\begin{enumerate}[(i)]
		\item $ [X_0,X_i]=X_{i+1}$ for $ i=1,\dots,s-1 $, then $ \G $ is said to be \emph{of the first kind}.
		\item  $ [X_0,X_i]=X_{i+1} $ for $ i=1,\dots,s-2 $ and $ [X_i,X_{s-i}]=(-1)^iX_s $ for  $ i=1,\dots,s-1  $, then $ \G $ is said to be \emph{of the second kind}.
	\end{enumerate}
\end{defn}
\noindent
In what follows, any basis of the Lie algebra $ \{X_0,\dots,X_s\} $ satisfying relations (i) or (ii) will be called \emph{filiform basis}.
The classification of filiform groups below follows from \cite[Proposition 5]{Vergne}.
\begin{prop}[Vergne]\label{prop:classif_of_filiforms}
	Let $ \G $ be a filiform group of step $ s $. If $ s = 2n+1 $  for $ n\geq 2 $, then $ \G $ is either of the first kind or of the second kind. Otherwise $ \G $ is of the first kind.
\end{prop}
\noindent
We point out that our choice of basis for the filiform groups of the second kind differs from the one of Vergne. Indeed, in \cite{Vergne} the author considers a basis $ \{Y_0,\dots,Y_s\} $ for which $ [Y_0,Y_i]=Y_{i+1} $ and $ [Y_i,Y_{s-i}]=(-1)^iY_s $ for $ i=1,\dots,s-1  $. One sees that the basis of Vergne has an extra nonzero bracket $ [Y_0,Y_{s-1}]=Y_s $ and the basis in Definition \ref{def:filiforms_first_and_second} (ii) is obtained by choosing $ X_0 = Y_0 + Y_1 $ and $ Y_i = X_i $ for $ i= 1,\dots,s $. In addition to having fewer nontrivial bracket relations, our choice of basis has the benefit that it is adapted to the abnormal lines, as we will see next.

\begin{prop}\label{prop:abn_lines_filiform}
	Let $ \G $ be a filiform group of step at least $ 3 $ and let $ \{X_0,\dots,X_s \} $ be a filiform basis for $ \g $. Then the line $t\mapsto \exp(tX_1) $ is abnormal. If $ \G $ is of the first kind, then this is the only abnormal horizontal line of $ \G $. Otherwise there exists exactly one other abnormal horizontal line, namely $ t\mapsto \exp(tX_0) $.
\end{prop}
\begin{proof}
	We begin by proving the following claim: if $ \G $ is a filiform group of step $ s \geq 3 $ and $ X \in \g_1 $, then the line $ t\mapsto \exp(tX) $ is abnormal if and only if $ [X,X_i]=0 $ for some $ i= 2,\dots,s-1 $. Assume first that $ [X,X_i]=0 $ for some  $ i= 2,\dots,s-1 $. Then, since $ \g_i $ is one-dimensional, $ [X,\g_i] = 0 $ and in particular 
	\[ 
	\g_{i+1}\cap \mathrm{span}\{ \mathrm{ad}_X^k(\g_1):k=0,\dots,s-1\}=\mathrm{ad}^i_X(\g_1) = [X,\mathrm{ad}^{i-1}_X(\g_1)] = 0.
	\]
	 Hence, by Proposition \ref{prop:nonabnormal}, the line $ t\mapsto \exp(tX) $ is abnormal. To prove the other implication, suppose that, for each $ i = 2,\dots, s-1 $, one has $ [X,X_i]\neq 0 $. Since the first layer $ \g_1 $ does not contain elements of the center, we also find $ Y_1 \in \g_1 $ for which $ [X,Y_1]\neq 0 $. Observe now that, since $ \g_i = \mathrm{span}([X,X_{i-1}]) $, one has $ \g_i = \mathrm{span}(\mathrm{ad}^{i-1}_X(Y_1)) $ for all $ i = 2,\dots, s $. By Proposition \ref{prop:nonabnormal}, the claim is proved.
	
	As a consequence of the claim, since any filiform basis satisfies $ [X_1,X_2] = 0 $, the line $t\mapsto \exp(tX_1) $ is abnormal. If $ \G $ is of the second kind, then the relation $[X_0,X_{s-1}]= 0$ shows that the line $  t\mapsto \exp(tX_0) $ is abnormal as well. We now show that, if $ \G $ is of the first kind, the line  $t\mapsto \exp(tX_1) $ is the only abnormal horizontal line of $ \G $. Indeed, for every $ a\in \R $, the element $ X \coloneqq X_0 + aX_1 \in \g_1 $ satisfies $ [X,X_i]= X_{i+1} $ for all $ i = 1,\dots,s-1 $. By the previous observations, the line $ t \mapsto \exp(tX) $ is non-abnormal.
	
	We are left to prove that a filiform group cannot possess more than two abnormal horizontal lines. This fact would imply that $  t\mapsto \exp(tX_0) $ and $t\mapsto \exp(tX_1) $ are the only horizontal abnormal lines of the filiform groups of the second kind. First, notice that a line $ t\mapsto \exp(tX) $ is abnormal if and only if $ X \in \bigcup_{i=2}^{s-1}\ker \phi_i $, where $ \phi_i \colon \g_1 \to \g_{i+1} $ is defined by $ \phi_i(X) = [X,X_i] $. According to Proposition \ref{prop:classif_of_filiforms}, the algebra $ \g/\g_s $ is filiform of step $s-1$ and therefore is of the first kind. By the previous part of this proof, the set $ \bigcup_{i=2}^{s-2}\ker \phi_i \subseteq  \g/\g_s$ is one-dimensional. Since $ \phi_{s-1} $ is surjective, also $ \dim \ker \phi_{s-1} = 1 $ and we conclude that $ \bigcup_{i=2}^{s-1}\ker \phi_i $ cannot contain more than two linearly independent lines.
\end{proof}
\begin{corollary}
	Every filiform group has a horizontal non-abnormal line. In particular, Corollary~\ref{cor:LR} applies to all filiform groups.
\end{corollary}
\begin{rmk}
	Applying \cite[Proposition 2.21]{LDMOPV}, one can check that in filiform groups of the first kind the only horizontal injective abnormal curve from the origin is, up to reparametrization, the line $ t \mapsto \exp(tX_1) $. On the other hand, in filiform groups of the second kind there are also horizontal abnormal curves that are not lines. For example, the curve defined by
	\[
	t \mapsto \begin{cases}
	\exp(tX_1)& \text{ for } t\in[0,1], \\
	\exp(X_1)\exp((t-1)X_0)& \text{ for } t\in [1,2]
	\end{cases}
	\] 
	is abnormal.
\end{rmk}
\noindent
 For the sake of completeness, we end by describing the graded Lie algebra automorphisms (i.e., the stratification preserving Lie automorphisms) of filiform Lie algebras. By the following two propositions, we observe that any linear bijection on the horizontal layer that fixes the abnormal lines extends uniquely to a Lie algebra automorphism. 
\begin{prop}\label{prop:automs_first_kind}
	Let $ \g $ be a filiform Lie algebra of the first kind of step at least $ 3 $ equipped with a filiform basis $\{X_0,\dots,X_s\} $. The linear transformation on $ \g_1 $ that in basis $ \{X_0,X_1\} $ is given by the matrix	
	\begin{align*}
	\begin{pmatrix}
	a & 0 \\
	c & b
	\end{pmatrix}
	\end{align*}
	induces a (graded) Lie algebra automorphism for every $ a,b\in \R \setminus\{0\} $ and  $ c\in \R $. Moreover, every graded Lie algebra automorphism of $ \g $ is of this form.
\end{prop}
\begin{proof}
	Notice that  the only horizontal vectors that commute with $ \g_2 $ are those parallel to $ X_1 $. Therefore, any $ \psi \in \mathrm{Aut}(\g) $ maps $ \psi(X_1) = bX_1 $ with $ b \in \R\setminus\{0\} $. Then mapping $ \psi(X_0) = aX_0 + cX_1 $ defines a Lie algebra automorphism for any choice of $ a,b \in \R \setminus\{0\} $ and $ c\in \R $ by $ \psi(X_i) \coloneqq \mathrm{ad}_{\psi_1(X_0)}^{i-1}(\psi(X_1)) = a^ibX_i$.
\end{proof}
\begin{prop}\label{prop:automs_second_kind}
	Let $ \g $ be a filiform Lie algebra of the second kind with a filiform basis $\{X_0,\dots,X_s\} $. The linear transformation on $ \g_1 $ that in basis $ \{X_0,X_1\} $ is given by the matrix	
	\begin{align*}
	\begin{pmatrix}
	a & 0 \\
	0 & b
	\end{pmatrix}
	\end{align*}
	induces a (graded) Lie algebra automorphism for every $ a,b\in \R \setminus\{0\} $. Moreover, every graded Lie algebra automorphism of $ \g $ is of this form.
\end{prop}
\begin{proof}
	Similarly to the filiform groups of the first kind, here $ X_1 $ is the unique direction commuting with $ \g_2 $, and hence the line $ bX_1 $, $ b\in \R $, must be fixed by every $ \psi \in \mathrm{Aut}(\g) $. Since in addition a graded Lie algebra automorphism maps abnormal horizontal lines into abnormal horizontal lines, necessarily $ \psi(X_0)= aX_0 $ for some $ a\in \R \setminus\{0\} $ according to Proposition \ref{prop:abn_lines_filiform}. Let us verify that the linear map defined by $ \psi(X_0)=aX_0 $ and $ \psi(X_1)=bX_1 $ on $ \g_1 $ induces a Lie algebra automorphism for all $a, b \in \R\setminus\{0\}$ by explicitly calculating the bracket relations. Indeed, the extension
	\begin{align*}
	\psi(X_i) &= \psi(\mathrm{ad}^{i-1}_{X_0}X_1) \coloneqq \mathrm{ad}^{i-1}_{\psi(X_0)}\psi(X_1) = a^{i-1}bX_i
	\quad\forall\; i = 1,\dots,s-1, \\
	\psi(X_s) &= \psi(-[X_1,X_{s-1}]) \coloneqq -[\psi(X_1),\psi(X_{s-1})] = a^{s-2}b^2X_s
	\end{align*}
	satisfies
	\[
	[\psi(X_i),\psi(X_{s-i})] = [a^{i-1}bX_i,a^{s-i-1}bX_{s-i}] = a^{s-2}b^2(-1)^{i}X_s = \psi([X_i,X_{s-i}]) \quad\forall\; i = 2,\dots,s-2
	\]
	and all the other brackets are zero, as required.
\end{proof}
\begin{rmk}
	To the best of our knowledge, in the literature (see e.g.\ \cite[page 93]{Vergne} or \cite[page 49]{Pansu}) there is no clear motivation why the two types of filiform groups are not isomorphic; let us provide some evidence here.
	First of all, one can check that, in dimension 4, there is only one filiform group, known as the {\em Engel group}.
	Starting from dimension 6, we provided two ways to distinguish the two classes.
	A first reason is that the spaces of graded automorphisms are different: in filiform groups of the first kind this class is 3 dimensional, while for
	the second kind this class is 2 dimensional.
	A second reason, which has a control-theoretic flavor, is that the two classes have different number of abnormal one-parameter subgroups:  in filiform groups of the first kind there is only one, 
	while for the second kind there are two.
	Recall that the stratification of a stratifiable group is unique up to isomorphisms, see \cite[Proposition 2.17]{LeDonne}. Hence, since we showed that the two classes are different as stratified groups, they are different as Lie groups.
\end{rmk}
\begin{rmk}
	We stress that the type $\star$ condition, see \cite{Marchi}, and the non-abnormality condition introduced in Corollary \ref{cor:LR} are independent. Indeed, all step-2 Carnot groups are of type $\star$, but not all step-2 Carnot groups have horizontal non-abnormal lines. Consider for example the free Carnot group $\mathbb F_{3,2}$ of step 2 and rank 3. One can easily check that, in this case, all horizontal lines are abnormal. On the other hand, as shown in Proposition \ref{prop:abn_lines_filiform}, all filiform groups have nontrivial horizontal non-abnormal lines but already the Engel group $\mathbb E$ does not satisfy the type $\star$ condition.
\end{rmk}
\begin{rmk}
	We stress that no filiform group is Pansu-rigid, as instead stated in \cite[page 49]{Pansu}.
	Indeed, in both types of filiform groups, there are more graded automorphisms than just the homotheties, in which case we would have said, by definition, that the Carnot group is Pansu-rigid, see \cite{LeDonne_Ottazzi_Warhurst}.
	It is however true that, in filiform groups of the second kind, the only graded  automorphisms that are unipotent are the dilations.
\end{rmk}

\bibliographystyle{alpha}
\bibliography{ConeBib}
\end{document}